\documentclass[11pt,oneside,a4paper]{amsart}
\usepackage{graphicx}
\usepackage{amsthm}
\usepackage{hyperref}
\usepackage{amsmath}
\usepackage{amssymb}
\usepackage[T1]{fontenc}
\usepackage{color}

\usepackage{mathrsfs}

\newtheorem{lem}{Lemma}[section]
\newtheorem{thm}[lem]{Theorem}
\newtheorem{rem}[lem]{Remark}
\newtheorem{defin}[lem]{Definition}
\newtheorem{cor}[lem]{Corollary}
\newtheorem{example}[lem]{Example}
\newtheorem{prop}[lem]{Proposition}

\def \NN{\mathbb{N}}

\def \RR{\mathbb{R}}
\def \Rd{{\RR^d}}

\newcommand{\ind}{{\bf 1 }}

\title[Sharp Gaussian estimates]{Sharp Gaussian estimates for Schr{\"o}dinger heat kernels: $L^p$ integrability conditions}


\author{Krzysztof Bogdan}
 \address{Wroc{\l}aw University of Science and Technology,
Wybrze{\.z}e Wyspia{\'n}skiego 27,
50-370 Wroc{\l}aw, Poland}
\email{bogdan@pwr.edu.pl}

\author{Jacek Dziuba{\'n}ski}
\address{University of Wroc{\l}aw,
Pl. Grunwaldzki 2/4,
50-384 Wroc{\l}aw,
Poland}
\email{Jacek.Dziubanski@math.uni.wroc.pl}

\author{Karol Szczypkowski}
\address{Universit{\"a}t Bielefeld, Postfach 10 01 31,
D-33501 Bielefeld, Germany  and
Wroc{\l}aw University of Science and Technology,
Wybrze{\.z}e Wyspia{\'n}skiego 27,
50-370 Wroc{\l}aw, Poland
}
\email{karol.szczypkowski@math.uni-bielefeld.de, karol.szczypkowski@pwr.edu.pl}


\begin{document}

\begin{abstract}
We give new sufficient conditions for comparability
of
the fundamental solution of the Schr{\"o}dinger equation $\partial_t=\Delta+V$ with the  Gauss-Weierstrass kernel and
show that local $L^p$
integrability of $V$ for $p> 1$ is not necessary for the comparability.
\end{abstract}

\maketitle

\section{Introduction and Preliminaries}\label{sec:prel}
Let $d=1,2,\ldots$.
We consider the Gauss-Weierstrass kernel,
\[g(t,x,y)=
(4\pi t)^{-d/2} e^{-|y-x|^2/(4t)}, \qquad t>0,\ x,y\in\Rd.\]
It is well known that $g$ is a time-homogeneous probability transition density.
For a
function $V$
we let $G$
be the Schr\"odinger perturbation of $g$ by $V$, i.e., the fundamental solution of $\partial_t=\Delta+V$, 
determined by the following Duhamel or perturbation formula for $t>0$, $x,y\in \Rd$,
\[
G(t,x,y)=g(t,x,y)+\int_0^t \int_\Rd G(s,x,z)V(z)g(t-s,z,y)dzds.
\]
Under appropriate assumptions on $V$, the definition of $G$ may be given by the Feynmann-Kac formula \cite[Section~6]{MR2457489}, the Trotter formula \cite[p.~467]{MR1978999}, the perturbation series \cite{MR2457489}, or by means of quadratic forms on $L^2$ spaces \cite[Section~4]{MR591851}.
In particular the assumption $V\in L^p(\Rd)$ with $p>d/2$ was
used
by Aronson
\cite{MR0435594}, Zhang \cite[Remark~1.1(b)]{MR1978999} and by Dziuba{\'n}ski and Zienkiewicz \cite{MR2164260}.
Aizenman and Simon \cite{MR644024,MR670130} proposed
functions $V(z)$ from the Kato class, which
contains
$L^{p}(\Rd)$
for every $p>d/2$
\cite[Chapter~4]{MR644024}, \cite[Chapter 3, Example 2]{MR1329992}.
An enlarged Kato class was used by Voigt \cite{MR845197} in the study of Schr{\"o}dinger semigroups on  $L^1$ \cite[Proposition~5.1]{MR845197}.
For time-dependent perturbations $V(u,z)$, Zhang \cite{MR1360232,MR1488344} introduced the so-called parabolic Kato condition.
It was then
generalized and employed by Schnaubelt and Voigt \cite{MR1687500}, Liskevich and Semenov
\cite{MR1783642}, Liskevich, Vogt and Voigt
\cite{MR2253015}, and Gulisashvili and van Casteren \cite{MR2253111}.

We say that $G$ has sharp Gaussian estimates if $G$ is comparable with $g$, at least in bounded time (see below for details).
A
sufficient condition for the sharp Gaussian estimates was given
by Zhang in \cite{MR1978999}.
As noted in \cite[Remark~1.2(c)]{MR1978999}, the condition
may be stated in terms of the {\it bridges} of $g$.
Bogdan, Jakubowski and Hansen \cite[Section~6]{MR2457489} and Bogdan, Jakubowski and Sydor \cite{MR3000465} gave analogous 
conditions for general transition densities.

Given a real-valued Borel measurable function $V$ on $\Rd$ we ask if there are numbers $0<c_1\le c_2<\infty$ such that the following two-sided bounds hold,
\begin{align}\label{est:sharp_uni}
c_1  \leq \frac{G(t,x,y)}{g(t,x,y)}\leq c_2,\qquad t>0,\ x,y\in \Rd.
\end{align}
We also ponder a weaker property--if for a given $T\in (0,\infty)$,
\begin{align}\label{est:sharp_time}
c_1  \leq \frac{G(t,x,y)}{g(t,x,y)}\leq c_2 \,,\qquad 0<t\le T,\ x,y\in \Rd.
\end{align}
We call \eqref{est:sharp_uni} and \eqref{est:sharp_time} {\it sharp Gaussian estimates} or bounds, respectively {\it global} and {\it local} in time.
We observe that the inequalities in \eqref{est:sharp_uni} and \eqref{est:sharp_time} are stronger than
plain
{\it 
Gaussian estimates:}
\begin{align*}\label{e:Ge}
c_1\, (4\pi t)^{-d/2} e^{-\frac{|y-x|^2}{4t\varepsilon_1}} \leq G(t,x,y)\leq c_2\, (4\pi t)^{-d/2} e^{-\frac{|y-x|^2}{4t\varepsilon_2}},
\end{align*}
where
$0<\varepsilon_1,c_1\le 1\le \varepsilon_2,c_2<\infty$. We note in passing
that
Berenstein proved
the Gaussian estimates for  $V\in L^p$ with $p>d/2$ (see \cite{MR1642818}),
Simon \cite[Theorem B.7.1]{MR670130} resolved
them for $V$ in the Kato class, Zhang  used subparabolic Kato class for the same end in \cite{MR1488344} 
and the so-called 4G inequality
was used by Bogdan and Szczypkowski in \cite{MR3200161}.
For further discussion
we refer the reader to
\cite{MR1783642}, \cite{MR2253015}, \cite{MR1994762}, \cite{MR1978999}, and the Introduction in \cite{MR3200161}.

It is difficult to explicitely characterize
\eqref{est:sharp_uni} and  \eqref{est:sharp_time},
especially for those $V$ that take on positive values.
Arsen'ev proved \eqref{est:sharp_time} for $V\in L^p+L^{\infty}$
with $p>d/2$, $d\geq 3$, and van Casteren \cite{MR1009389} proved
it for $V$ in the intersection of the Kato class and
$L^{d/2}+L^{\infty}$ for $d\geq 3$ (see \cite{MR1994762}). Arsen'ev also obtained \eqref{est:sharp_uni} for $V\in L^p$ with $p > d/2$ under additional smoothness assumptions (see \cite{MR1642818}).
Zhang \cite[Theorem~1.1]{MR1978999} and Milman and Semenov \cite[Theorem~1C, Remark (2)]{MR1994762}
gave  sufficient supremum-integral conditions for \eqref{est:sharp_time} and \eqref{est:sharp_uni} for signed $V$  and characterized \eqref{est:sharp_uni}
for $V\le 0$. Their results
left
open
certain natural questions about the  class
of admissible functions $V$, especially for dimensions $d\ge 4$.
We
were particularly motivated by the question of 
Liskevich and Semenov about 
the connection of the sharp Gaussian estimates,
the potential-boundedness and the 
$L^{d/2}$ integrability condition, cf.
\cite[Remark (3), p. 602]{MR1642818}. In this work we use potential-boundedness \eqref{e:si} and bridges potential-boundedness of $V$ to study the connection of the sharp Gaussian estimates and the $L^p$ integrability, disregarding the Kato condition.
In Theorem~\ref{prop:most_extended} below
we
give new sufficient conditions for the sharp Gaussian estimates,
which help
verify
that
$L^p$ integrability is not necessary for \eqref{est:sharp_uni} or \eqref{est:sharp_time}. Namely,  for $d\ge 3$ we present in Corollary~\ref{cor:ce} functions $V$ such that
\eqref{est:sharp_uni} holds but $V\notin L^1(\Rd)\cup \bigcup_{p>1}L^p_{loc}(\Rd)$. Our examples are highly anisotropic because they are constructed from tensor products, and to study them we crucially use factorization of the Gauss-Weierstrass kernel.
Before discussing the present results we should mention our more recent paper \cite{2016arXiv160603745B}, where we give a new characterization of \eqref{est:sharp_uni} and resolve the question of  Liskevich and Semenov. Both papers grew out from our work on this question but contain
different observations.
In fact \cite{2016arXiv160603745B} uses the preliminary results stated in this Introduction, apart from which the two papers  have no overlap.

The structure of the remainder of the paper is the following.
Below in this section we give definitions and preliminaries,
and organize the relevant results 
existing in the literature.
In particular in Lemma~\ref{cor:gen_neg}
we 
present characterizations of \eqref{est:sharp_uni} and \eqref{est:sharp_time}
for $V\le 0$. In Theorem~\ref{prop:most_extended}
 of Section~\ref{sec:gen_app} we propose new sufficient conditions for
\eqref{est:sharp_uni} and  \eqref{est:sharp_time}, with emphasis on
those functions $V$ which factorize as tensor products.
In Section~\ref{sec:e} we prove Corollary~\ref{cor:ce} 
and give examples which illustrate and comment our results.
In Section~\ref{sec:a} we give supplementary details.

Let $\NN=\{1,2,\ldots\}$, $f^+=\max\{0,f\}$ and $f^-=\max\{0,-f\}$. All the considered functions $V\colon\Rd\to
[-\infty, \infty]$ are assumed
Borel measurable.

Here is a quantity to characterize \eqref{est:sharp_uni} and  \eqref{est:sharp_time}:
\begin{align*}
S(V,t,x,y)=\int_0^t \int_{\Rd} \frac{g(s,x,z)g(t-s,z,y)}{g(t,x,y)}|V(z)|\,dzds, \quad t>0,\ x,y\in \Rd.
\end{align*}
In what follows we often abbreviate $S(V)$.
The motivation
for using $S(V)$ comes from
\cite[Lemma~3.1 and Lemma~3.2]{MR1978999} and from \cite[(1)]{MR2457489}.

In the next two results we
compile  \cite[Theorem~1.1]{MR1978999} and observations from \cite{MR2457489} and \cite{MR3000465} to give conditions for the sharp Gaussian estimates.
For completeness, the proofs are given in Section~\ref{sec:a}.

\begin{lem}\label{cor:gen_neg}
If $V\leq 0$, then \eqref{est:sharp_uni} is equivalent to
\begin{equation}\label{e:sSbi}
\sup_{t>0,\,x,y\in\Rd} S(V,t,x,y)<\infty.
\end{equation}
If $V\leq 0$, then for every $T\in (0,\infty)$,
\eqref{est:sharp_time} is equivalent to
\begin{equation}\label{e:sSbt}
\sup_{0<t\leq T,\,x,y\in\Rd} S(V,t,x,y)<\infty.
\end{equation}
\end{lem}
It is appropriate to say that $V$ satisfying \eqref{e:sSbi} or \eqref{e:sSbt} has bounded potential for bridges (is bridges~potential-bounded), globally or locally in time, respectively, cf. Section~~\ref{sec:gen_app}.

\begin{lem}\label{lem:gen_neg}
If for some
$h>0$ and $0\le \eta<1$ we have
$$\sup_{0<t\leq h,\,x,y\in\Rd} S(V^+,t,x,y)\le \eta,$$
and $S(V^-)$ is bounded on bounded subsets of
$(0,\infty)\times\Rd\times\Rd$
then
\begin{align}\label{gen_est}
e^{-S(V^-,t,x,y)} \leq \frac{G(t,x,y)}{g(t,x,y)}\leq \left(\frac{1}{1-\eta}\right)^{1+t/h}, \qquad t>0, \ x,y\in \Rd \,.
\end{align}
\end{lem}

We
record
the following observations on
integrability
and on the potential-boundedness \eqref{e:si} of  functions $V$ which are bridges potential-bounded.
\begin{lem}\label{l:b}
If $S(V,t,x,y)<\infty$ for some $t>0$, $x,y\in\Rd$, then $V\in L^1_{\rm loc}(\Rd)$.
If \eqref{e:sSbt} holds, then
\begin{align}\label{e:si_local}
\sup_{x\in\Rd}\int_0^T \int_{\Rd} g(s,x,z)|V(z)|\,dzds<\infty\,.
\end{align}
If  \eqref{e:sSbi} even holds,
then $|V|$ has bounded Newtonian potential:
\begin{equation}\label{e:si}
\sup_{x\in\Rd}\int_0^{\infty} \int_{\Rd} g(s,x,z)|V(z)|\,dzds<\infty\,.
\end{equation}
If $V \geq 0$, then \eqref{est:sharp_uni} implies 
\eqref{e:sSbi} and \eqref{est:sharp_time} implies
\eqref{e:sSbt}.
If $d=3$ and $V\le 0$, then
\eqref{e:si} is equivalent to \eqref{est:sharp_uni}.
\end{lem}
\begin{proof}
The first statement follows, because
$g(t,x,y)$ is locally bounded from below on $(0,\infty)\times\Rd\times\Rd$ (see
\cite[Lemma~3.7]{GS} for a
quantitative general result).
Since
$\int_\Rd S(V,t,x,y)g(t,x,y)\,dy=\int_0^t \int_{\Rd} g(s,x,z)|V(z)|\,dzds$,
we see that
\eqref{e:sSbt} implies \eqref{e:si_local}
and
\eqref{e:sSbi} implies \eqref{e:si}.
The next to the last sentence easily follows from Duhamel formula and the fact that $G\ge g$ in this case.
The last statement follows from \cite[Remark~(2) and (3) on p. 4]{MR1994762}.
\end{proof}
We note that \eqref{e:si} and thus also \eqref{e:sSbi}
fail for all nonzero $V$
in dimensions $d=1$ and $d=2$, because then
$\int_0^{\infty} g(s,x,z)ds\equiv \infty$.
Consequently,
\eqref{est:sharp_uni} fails for
nontrivial $V\le 0$ if $d=1$ or $2$.
For $d\geq 3$ we let $C_d=\frac{\Gamma(d/2-1)}{4\pi^{d/2}}$.
 The
Newtonian potential  of nonnegative function $f$ and $x\in \Rd$ will be denoted
\begin{align*}
-\Delta^{-1} f(x)
:=\int_0^{\infty}\int_{\Rd}g(s,x,z) f(z)\,dzds
=C_d\int_{\Rd} \frac{1}{|z-x|^{d-2}}f(z)\,dz\,.
\end{align*}
Thus,
\eqref{e:si} reads $\|\Delta^{-1} |V| \|_{\infty}<\infty$.
We also note that by Theorem~1C~(2), Remark~(3) on p.~4, and the comments before Theorem~1B in \cite{MR1994762}, $\|\Delta^{-1} V \|_{\infty}<1$
suffices for \eqref{est:sharp_uni} when $d=3$ and $V\ge 0$.

\section{Sufficient conditions for
the sharp Gaussian estimates}\label{sec:gen_app}
Recall from \cite[(2.5)]{MR549115} that for $p\in [1,\infty]$,
\begin{align*}
\| P_t f\|_{\infty} \leq C(d,p) \,t^{-d/(2p)} \|f \|_p\, ,\qquad t>0\,,
\end{align*}
where $P_t f(x)=\int_{\Rd} g(t,x,z)f(z)dz$, $f\in L^p(\Rd)$ and
$$
C(d,p)=\begin{cases} (4\pi)^{-d/2}, \quad & \mbox{ if \ \ } p=1\,,\\ (4\pi)^{-d/(2p)} (1-p^{-1})^{(1-p^{-1})d/2}, & \mbox{ if \ \ } p\in (1,\infty]\,. \end{cases}
$$
We will give an analogue for the {\it bridges} $T^{t,y}_s$. Here $t>0$, $y\in\Rd$, and
$$
T^{t,y}_s f (x) = \int_{\Rd} \frac{g(s,x,z)\,g(t-s,z,y)}{g(t,x,y)} f(z) \,dz\,,\qquad 0<s<t,\quad x\in\Rd\,.
$$
Clearly,
\begin{equation}\label{e:sT}
T^{t,y}_s f (x) =
T^{t,x}_{t-s} f (y),\qquad 0<s<t,\quad x,y\in\Rd\,.
\end{equation}
By  the Chapman-Kolmogorov equations (the semigroup property) for the kernel $g$, we have $T^{t,y}_s 1=1$.
We also note that $S(V)$ is related to the potential ($0$-resolvent) operator of $T$ as follows,
$$
S(V,t,x,y)=\int_0^t T^{t,y}_s\, |V|(x)\,ds\,.
$$
\begin{lem}
\label{lem:Lp}
For $p\in [1,\infty]$ and $f\in L^p(\Rd)$ we have
\begin{align*}
\| T^{t,y}_s f\|_{\infty} \leq C(d,p) \,\left[\frac{(t-s)s}{t}\right]^{-d/(2p)} \|f \|_p\,,\qquad 0<s<t,\,y\in\Rd\,.
\end{align*}
\end{lem}
\begin{proof}
We note that
$$\frac{g(s,x,z)\,g(t-s,z,y)}{(4\pi)^{-d/2}g(t,x,y)}=\left[\frac{(t-s)s}{t}\right]^{-\frac{d}{2}}
\exp\left[ -
\frac{|z-x|^2}{4s} -\frac{|y-z|^2}{4(t-s)} +\frac{|y-x|^2}{4t}
\right].
$$
As in \cite[(3.4)]{MR1457736}, we have
\begin{equation}\label{e:ti}
\frac{|z-x|^2}{4s}+\frac{|y-z|^2}{4(t-s)}\ge \frac{|y-x|^2}{4t}.
\end{equation}
Indeed, (\ref{e:ti}) obtains from by the triangle  and Cauchy-Schwarz inequalities:
\begin{align*}
|y-x|&\le \sqrt{s}\frac{|z-x|}{\sqrt{s}}+\sqrt{t-s}\frac{|y-z|}{\sqrt{t-s}}
\le \sqrt{t}\left(
\frac{|z-x|^2}{s}+\frac{|y-z|^2}{t-s}
\right)^{1/2}.
\end{align*}
For $p=1$, the assertion of the lemma results from \eqref{e:ti}.
For $p\in (1,\infty)$, we let $p'=p/(p-1)$, apply H{\"o}lder's inequality and the semigroup property, and by the first case we obtain
 \begin{equation*}\begin{split}
|T_s^{t,y}f(x)| &  \leq g(t,x,y)^{-1}\Big[\int g(s,x,z)^{p'}g(t-s,z,y)^{p'}\, dz\Big]^{1\slash p'}\| f\|_p \\
  &=C(d,p)\Big[\frac{s(t-s)}{t}\Big]^{-d\slash (2p)}\| f\|_p.
\end{split}\end{equation*}
Here we also use the identity
$g(s,x,z)^{p'}=g(s/p',x,z)(4\pi s)^{(1-p')d/2}(p')^{-d/2}$.
For $p=\infty$, the assertion  follows from the identity $T^{t,y}_s 1=1$.
\end{proof}

Zhang \cite[Proposition~2.1]{MR1978999}
showed that \eqref{est:sharp_uni} and \eqref{est:sharp_time} hold for $V$ in specific $L^p$ spaces (see also \cite[Theorem~1.1 and Remark~1.1]{MR1978999}).
We can reprove his result as follows.
\begin{prop}\label{prop:Zhang_most}
Let $V\colon\Rd\to\RR$ and $p,q\in [1,\infty]$.
\begin{enumerate}
\item[\rm (a)]
If $V\in L^p(\Rd)$, $p> d/2$ and $c= C(d,p) \frac{[\Gamma(1-d/(2p))]^2}{\Gamma(2-d/p)} \|V \|_p$, then
\begin{align*}
\sup_{x,y\in\Rd} S(V,t,x,y)
\leq c\,t^{1-d/(2p)}\,, \qquad t>0\,.
\end{align*}
\item[\rm(b)] If $V\in L^p(\Rd)\cap L^q(\Rd)$ and $q<d/2<p$, then \eqref{e:sSbi} holds.
\end{enumerate}
\end{prop}
\begin{proof}
Part ${\rm(a)}$ follows from Lemma~\ref{lem:Lp}, so we proceed to ${\rm(b)}$.
For $t>2$,
\begin{align}\label{e:dT}
\int_0^t T^{t,y}_s |V|(x)\,ds = \int_0^{t/2} T^{t,y}_s |V|(x)\,ds+\int_0^{t/2} T^{t,x}_s |V|(y)\,ds\,.
\end{align}
Estimating the first term of the sum, by
Lemma~\ref{lem:Lp} we obtain
\begin{align}
\int_0^{t/2}  \!\!\!\!T^{t,y}_s |V|(x)\,ds
&\leq c\, \| V \|_p \int_0^1 \left[\frac{(t-s)s}{t}\right]^{-d/(2p)} \!\!\!\! \!\!\!\!\!\!\!\!ds+c\, \| V \|_q\int_1^{t/2} \left[\frac{(t-s)s}{t}\right]^{-d/(2q)}   \!\!\!\!\!\!\!\!\!\!\!\!ds
\nonumber
\\
&\leq  c'\, \| V \|_p \int_0^1 s^{-d/(2p)} ds+c'\, \| V \|_q\int_1^{\infty} s^{-d/(2q)}ds.
\label{eq:jpq}
\end{align}
By \eqref{e:sT}, the second term
has the same bound.
For $t\in (0,2]$ we use ${\rm(a)}$.
\end{proof}

By Lemma~\ref{cor:gen_neg} and \ref{lem:gen_neg} we get the following conclusion.
\begin{cor}\label{cor:esc}
Under the assumptions of
{\rm Proposition~\ref{prop:Zhang_most}(a)}, $G$ 
satisfies the
sharp local Gaussian bounds \eqref{est:sharp_time}. 
If $V\le 0$ and the assumptions of
{\rm Proposition~\ref{prop:Zhang_most}(b)}
hold, then $G$ has the sharp global Gaussian bounds \eqref{est:sharp_uni}.
\end{cor}

Recall that \cite[Theorem~2]{MR1642818} and \cite[Remark~(1) and (4) on p. 4]{MR1994762} yield \eqref{est:sharp_uni} for $d\geq 4$ if
$\|\Delta^{-1} V^- \|_{\infty}$ and $\|V^-\|_{d/2}$
are finite, $\|\Delta^{-1} V^+ \|_{\infty}<1$ and
$\|V^-\|_{d/2}$ is
small.
We can reduce
Proposition~\ref{prop:Zhang_most}${\rm(b)}$ to this result as follows.

\begin{lem}\label{lem:Zhang_Lisk_Sem}
The
assumptions of {\rm Proposition~\ref{prop:Zhang_most}${\rm(b)}$}
necessitate that
$d\geq 3$, $V\in L^{d/2}(\Rd)$ and $\|\Delta^{-1} |V| \|_{\infty}<\infty$.
\end{lem}
\begin{proof}
Plainly, the assumptions of {\rm Proposition~\ref{prop:Zhang_most}${\rm(b)}$} imply
$d>2$ and $V\in L^{d/2}(\Rd)$. We now verify that \mbox{$\|\Delta^{-1} |V| \|_{\infty}<\infty$.} By H\"older's inequality,
\begin{align*}
\sup_{x\in\Rd}\int_{B(0,1)} \frac{|V(z+x)|}{|z|^{d-2}} dz\leq \||z|^{2-d}\ind_{B(0,1)}(z) \|_{p'}\, \|V \|_{p}<\infty\,,\\
\sup_{x\in\Rd}\int_{B^c(0,1)} \frac{|V(z+x)|}{|z|^{d-2}} dz\leq \||z|^{2-d}\ind_{B^c(0,1)}(z) \|_{q'}\, \|V \|_{q}<\infty\,,
\end{align*}
where $p',q'$ are the exponents conjugate to $p,q$, respectively.
\end{proof}

In what follows, we propose suitable sufficient conditions for \eqref{est:sharp_uni} and \eqref{est:sharp_time}.
We let $d_1,d_2\in \NN$ and $d=d_1+d_2$.
\begin{rem}\label{rem:obs}
\rm
The Gauss-Weierstrass kernel $g(t,x)$ in $\Rd$ can be represented as a tensor product:
\begin{align*}
g(t,x)= (4\pi t)^{-d_1/2} e^{-|x_1|^2/(4t)}\, (4\pi t)^{-d_2/2} e^{-|x_2|^2/(4t)}\,,
\end{align*}
where $x_1\in\RR^{d_1}$, $x_2\in\RR^{d_2}$ and $x=(x_1,x_2)$.
The kernels of the bridges factorize accordingly:
\begin{align*}
&\frac{g(s,x,z)\,g(t-s,z,y)}{g(t,x,y)}\\
&=\frac{(4\pi s)^{-d_1/2} e^{-|z_1-x_1|^2/(4s)}(4\pi (t-s))^{-d_1/2} e^{-|y_1-z_1|^2/(4(t-s))}}{(4\pi t)^{-d_1/2} e^{-|y_1-x_1|^2/(4t)}}\\
&\times\frac{(4\pi s)^{-d_2/2} e^{-|z_2-x_2|^2/(4s)} (4\pi (t-s))^{-d_2/2} e^{-|y_2-z_2|^2/(4(t-s))}}{(4\pi t)^{-d_2/2} e^{-|y_2-x_2|^2/(4t)}}.
\end{align*}
\end{rem}

\begin{cor}\label{rem:prod}
Let $V_1\colon \RR^{d_1}\to \RR$, $V_2\colon\RR^{d_2}\to \RR$,
and $V(x)=V_1(x_1)V_2(x_2)$, where $x=(x_1,x_2)\in \Rd$, $x_1\in \RR^{d_1}$ and $x_2\in \RR^{d_2}$.
Assume that $V_1\in L^{\infty}(\RR^{d_1})$ and $\sup_{t>0,\,x_2,y_2\in\RR^{d_2}} S(V_2,t,x_2,y_2)<\infty$.
Then \eqref{e:sSbi} holds.
\end{cor}
\begin{proof}
In estimaing $S(V,t,x,y)$ we first use the factorization of the bridges and the boundedness of $V_1$, and then the Chapman-Kolmogorov equations and the boundedness of $S(V_2)$.
\end{proof}

Let
$p, p_1, p_2\in [1,\infty]$.
\begin{defin}\label{d:tp}
We write
$f\in L^{p_1}(\RR^{d_1})\times L^{p_2}(\RR^{d_2})$ if there are $f_1\in L^{p_1}(\RR^{d_1})$ and $f_2\in L^{p_2}(\RR^{d_2})$, such that
$$
f(x_1,x_2)=f_1(x_1) f_2(x_2)\,,\qquad x_1\in\RR^{d_1}, \ x_2\in\RR^{d_2}.
$$
\end{defin}
\vspace{5pt}
\noindent
Clearly, $L^{p}(\RR^{d_1})\times L^{p}(\RR^{d_2}) \subset
L^{p}(\RR^{d_1+d_2})$, in fact
$\|f\|_p=\|f_1\|_p\|f_2\|_p$ if $f (x_1,x_2)=f_1(x_1)f_2(x_2)$.
\begin{lem}
\label{lem:Lp1_Lp2}
For $f (x_1,x_2)=f_1(x_1)f_2(x_2) \in L^{p_1}(\RR^{d_1})\times L^{p_2}(\RR^{d_2})$, $0<s<t$ and $y\in\Rd$,
we have
\begin{align*}
\| T^{t,y}_s f\|_{\infty} \leq C(d_1,p_1)\, C(d_2,p_2) \left[\frac{(t-s)s}{t}\right]^{-d_1/(2p_1)-d_2/(2p_2)} \|f_1 \|_{p_1} \|f_2 \|_{p_2}\,.
\end{align*}
\end{lem}
\begin{proof}
We proceed as in the proof of Lemma~\ref{lem:Lp}, using Remark~\ref{rem:obs}.
\end{proof}
We
extend Proposition~\ref{prop:Zhang_most} as follows.
\begin{thm}\label{prop:most_extended}
Let $d_1, d_2\in \NN$, $d=d_1+d_2$, $V\colon \Rd\to \RR$,
$p_1, p_2\in [1,\infty]$
and
$$
\frac{d_1}{2p_1}+\frac{d_2}{2p_2}=1\,.
$$
\begin{enumerate}
\item[\rm (a)] If $r\in(p_1,\infty]$ and $V\in L^{r}(\RR^{d_1})\times L^{p_2}(\RR^{d_2})$, then
$$
\sup_{x,y\in\Rd} S(V,t,x,y)\leq c\, t^{1-d_1/(2r)-d_2/(2p_2)}\,,
$$
where $c=C(d_1,r)C(d_2,p_2) \frac{[\Gamma(1-d_1/(2r)-d_2/(2p_2))]^2}{\Gamma(2-d_1/r-d_2/p_2)} \|V_1 \|_{r}\|V_2 \|_{p_2}$.
\item[\rm(b)]
If $1\le q <p_1<r\le \infty$ and
$V\in \left[L^{q}(\RR^{d_1})\cap L^{r}(\RR^{d_1}) \right]\!\!\times L^{p_2}(\RR^{d_2})$, then
\eqref{e:sSbi} holds.
\end{enumerate}
\end{thm}
\begin{proof}
We follow
the proof of Proposition~\ref{prop:Zhang_most}, replacing Lemma~\ref{lem:Lp}
by Lemma~\ref{lem:Lp1_Lp2}.\end{proof}

By Lemma~\ref{cor:gen_neg} and \ref{lem:gen_neg} we get the following conclusion.
\begin{cor}\label{cor:escb}
Under the assumptions of
{\rm Theorem~\ref{prop:most_extended}(a)},
$G$ 
satisfies the sharp local Gaussian bounds \eqref{est:sharp_time}. If $V\le 0$ and the assumptions of
{\rm Theorem~\ref{prop:most_extended}(b)}
hold, then $G$ has the sharp global Gaussian bounds \eqref{est:sharp_uni}.
\end{cor}

Clearly, if $|U|\le |V|$, then $S(U)\le S(V)$. This may be used to
extend
the 
conclusions of Theorem~\ref{prop:most_extended} and Corollary~\ref{cor:escb}
beyond tensor products $V(x_1,x_2)=V_1(x_1)V_2(x_2)$.

\section{Examples}\label{sec:e}
Let $\ind_A$ denote the indicator function of $A$. In what follows, $G$ in \eqref{est:sharp_uni} is the Schr{\"o}dinger perturbation of $g$ by $V$.

\begin{example}\label{thm:Ld/2}
Let $d\geq 3$ and $1<p<\infty$.
For $x_1\in  \RR$, $x_2\in \RR^{d-1}$ we let
$V(x_1,x_2)=-
|x_1|^{-1/p}\ind_{|x_1|<1}\ind_{|x_2|<1}$.
Then \eqref{est:sharp_uni} holds but $V \notin L^{p}_{loc}(\Rd)$.
\end{example}
\noindent
Indeed,
$V(x_1,x_2)=V_1(x_1) V_2(x_2)$, where
\begin{align*}
V_1(x_1)&=-|x_1|^{-1/p} \ind_{|x_1|<1},\qquad x_1\in  \RR,\\
V_2(x_2)&=\ind_{|x_2|<1}, \qquad x_2\in \RR^{d-1}.
\end{align*}
Let
$$
1\le q<p_1<r<p,
$$
and
$$
p_2=\frac{d-1}{2}\frac{p_1}{p_1-1/2}.
$$
Since $d\ge 3$, $p_2>1$.
In the notation of Theorem~\ref{prop:most_extended} we have
\mbox{$d_1=1$}, \mbox{$d_2=d-1$}, and indeed
\mbox{$d_1/(2p_1)+d_2/(2p_2)=1$}.
Since $V_1\in L^r(\RR)\cap L^{q}(\RR)$ and $V_2\in L^{p_2}(\RR^{d-1})$,
the assumptions of Theorem~\ref{prop:most_extended}${\rm(b)}$ are satisfied,
and \eqref{est:sharp_uni} follows
by Corollary~\ref{cor:escb}.
Clearly, $V\notin L^{p}_{loc}(\Rd)$.

\begin{example}\label{ex:drugi}
For $d\geq 3$, $n=2,3,\ldots$, let $V_n(x)=|x_1|^{-1+1/n}\ind_{|x_1|<1} \ind_{|x_2|<1}$, where $x=(x_1,x_2)$, $x_1\in\RR$, $x_2\in \RR^{d-1}$. Let $a_n=\sup_{t>0,\, x,y\in\Rd} S(V_n,t,x,y)$,
$$
V(x)=-\sum_{n=2}^{\infty} \frac1{n^2}\frac{V_n(x)}{a_n}\,, \quad x\in \Rd.
$$
Then
\eqref{est:sharp_uni} holds but $V\notin \bigcup_{p> 1} L^p_{loc}(\Rd)$.
\end{example}
\noindent
Indeed,
$0<a_n<\infty$ by Example~\ref{thm:Ld/2}, and so
$$
\sup_{t>0,\,x,y\in\Rd} S(V,t,x,y)\leq
\sum_{n=2}^{\infty}\frac{1}{n^2} <\infty\,.
$$
This yields the global sharp Gaussian bounds.
For $p>1$ we let $m=\lceil \frac{p}{p-1} \rceil$, and we have
$m\geq 2$, $\frac{m-1}{m}p\geq 1$. Then,
$$
\int_{B(0,2)} |V(x)|^p\,dx \geq \left(\frac1{m^{2}a_m}\right)^p \int_{|x_2|<1}\int_{|x_1|<1} |x_1|^{ 
-\frac{m-1}{m} p
} \,dx_1 dx_2= +\infty\,.
$$

\begin{example}\label{ex:czwarty}
Let $d\ge 3$ and $V(x_1,x_2)=\frac{-1}{(|x_2|+1)^3}$ for $x_1\in\RR^{d-3}$, $x_2\in \RR^{3}$.
Then \eqref{est:sharp_uni} holds but $V\notin L^1(\RR^d)$.
\end{example}
\noindent
Indeed, we denote $V_2(x_2)=\frac{-1}{(|x_2|+1)^3}$, $x_2\in \Rd$, and by the symmetric rearrangement inequality \cite[Chapter~3]{MR1817225},
in dimension $d=3$ we have
\begin{align*}
0\le \Delta^{-1}V_2
\le C_{3}\int_{\RR^3} \frac{1}{|z|(|z|+1)^{3}} \,dz <\infty\,.
\end{align*}
By Lemma~\ref{l:b} and Lemma~\ref{cor:gen_neg}
$$
\sup_{t>0,\,x_2,y_2\in\RR^3} S(V_2,t,x_2,y_2)<\infty.
$$
By Corollary~\ref{rem:prod} and Lemma~\ref{cor:gen_neg} we see that \eqref{est:sharp_uni} holds. Clearly, $V\notin L^1(\RR^d)$.

\begin{cor}\label{cor:ce}
For every $d\ge 3$ there is a function $V$ such that \eqref{est:sharp_uni} holds but
$V\notin L^1(\Rd)\cup \bigcup_{p>1}L^p_{loc}(\Rd)$.
\end{cor}
\begin{proof}
Take the sum of the functions from
Example~\ref{ex:drugi} and Example~\ref{ex:czwarty}.
\end{proof}
We can have nonnegative examples, too. Namely, let $V\le 0$ be as in Corollary~\ref{cor:ce}.
Then $M=\sup_{t>0,x,y\in \Rd}S(V,t,x,y)<\infty$.
We
let $\tilde{V}= |V|/(M+1)$.
Then $\tilde{V}\ge 0$, $\tilde{V}\notin L^1(\Rd)\cup \bigcup_{p>1}L^p_{loc}(\Rd)$
and
$$
\sup_{t>0,\,x,y\in\Rd} S(\tilde{V},t,x,y)=M/(M+1)<1\,.
$$
Therefore \eqref{gen_est} holds for $\tilde{V}$ with $h=\infty$ and $\eta=M/(M+1)$, which  yields \eqref{est:sharp_uni}.

Let $d_1,d_2\in \NN$, $d=d_1+d_2$, $V_1\colon \RR^{d_1}\to \RR$, $V_2\colon\RR^{d_2}\to \RR$,
and $V(x_1,x_2)=V_1(x_1)+V_2(x_2)$, where $x_1\in \RR^{d_1}$ and $x_2\in \RR^{d_2}$.
Let $G_1(t,x_1,y_1)$, $G_2(t,x_2,y_2)$ be the 
Schr\"odinger perturbations of the Gauss-Weierstrass kernels on $\RR^{d_1}$ and $\RR^{d_2}$ by $V_1$ and $V_2$, respectively.
Then
$G(t,(x_1,x_2),(y_1,y_2)):= G_1(t,x_1,y_1) \allowbreak G_2(t,x_2,y_2)$ is the Schr\"odinger perturbation of the Gauss-Weierstrass kernel  on $\RR^d$ by $V$.
Clearly, if the sharp Gaussian estimates hold for $G_1$ and $G_2$, then they hold for $G$.
Our next example is aimed to show that such trivial conclusions are invalid for tensor products $V(x_1,x_2)=V_1(x_1)V_2(x_2)$.

\begin{example}\label{ex:nfS2}
Let $\varepsilon\in [0,1)$. For $x_1,x_2\in \RR^3$ let $V(x_1,x_2)= V_1(x_1)V_2(x_2)$, where
$$
V_1(x)=V_2(x)=-\frac{1-\varepsilon}{2}\ |x|^{-1-\varepsilon}\ \ind_{|x|<1}.
$$
Then the fundamental solutions in $\RR^3$ of
$\partial_t=\Delta+V_1$ and $\partial_t=\Delta+V_2$ satisfy \eqref{est:sharp_uni} and \eqref{est:sharp_time}, but 
that of $\partial_t=\Delta+V$
in $\RR^6$ 
satisfies neither \eqref{est:sharp_uni} nor \eqref{est:sharp_time}.
\end{example}
\noindent
Indeed, by the symmetric rearrangement inequality \cite[Chapter~3]{MR1817225}, 
\begin{align*}
0\le-\Delta^{-1}  V_1  (x)\leq
-\Delta^{-1} V_1  (0)=
\frac{1-\varepsilon}{8\pi} \int_{\{z\in\RR^3:|z|<1\}}\frac{1}{|z|} |z|^{-1-\varepsilon}\,dz= 1/2, 
\end{align*}
for all $x\in \RR^3$.
Thus, $\|\Delta^{-1}V_1\|_{\infty}=\| \Delta^{-1} V_2\|_{\infty}<\infty$. By
 Lemma~\ref{l:b} we get
\eqref{est:sharp_uni}
for the fundamental solutions  in $\RR^3$ of 
$\partial_t=\Delta+V_1$ and $\partial_t=\Delta+V_2$.
However, the fundamental solution in $\RR^6$ of
$\partial_t=\Delta+V$
fails \eqref{est:sharp_time}.
Indeed,
if we let $T\leq 1$, $a\in\RR^6$, $|a|=1$, and
$c=\int_0^1 p(s,0,a)ds$, then by \cite[Lemma~3.5]{MR1329992},
\begin{align*}
&\int_0^T \int_{\RR^6}  g(s,0,x)|V(x)|\,dxds
\geq
\int_{\{x \in\RR^6:|x|^2\leq T\}} \int_0^T g(s,0,x)ds\,|V(x)|\,dx\\
&\geq
c
\int_{\{x \in\RR^6:|x|^2\leq T\}}  \frac1{|x|^4} |V(x)|\,dx
\\
&\geq c \int_{\{x_1\in\RR^3:|x_1|^2<T/2\}}  |V_1(x_1)|
\int_{\{x_2\in\RR^3:|x_2|^2<T/2\}}
\frac{|V_2(x_2)|}{(|x_1|^2+|x_2|^2)^2} \,dx_2 dx_1
\\
&\geq  \frac{c (1-\varepsilon)}{2} \int_{\{x_1\in\RR^3:|x_1|^2<T/2\}} |V_1(x_1)|
\int_{\{x_2\in\RR^3:|x_2|^2<T/2\}}
\frac{|x_2|^{-1}}{(|x_1|^2+|x_2|^2)^2} \,dx_2 dx_1
\\
&= \frac{c (1-\varepsilon)}{2} \int_{\{x_1\in\RR^3:|x_1|^2<T/2\}}  |V_1(x_1)|
 \frac{\pi T}{|x_1|^2 (T/2+|x_1|^2)}\,dx_1 \\
&=  \pi^2 c\,T (1-\varepsilon)^2 \int_0^{\sqrt{T/2}} \frac{r^{-1-\varepsilon}}{T/2+r^2}\,dr=\infty\,.
\end{align*}
By  Lemma~\ref{l:b}, \eqref{e:sSbt} fails, and so does \eqref{est:sharp_time},
cf. Lemma~\ref{cor:gen_neg}.
Thus, the sharp Gaussian estimates may hold for the Schr\"odinger perturbations of the Gauss-Weierstrass kernels by $V_1$ and $V_2$ but fail for the Schr\"odinger perturbation of the Gauss-Weierstrass kernel by $V(x_1,x_2)=V_1(x_1)V_2(x_2)$. Considering $-V_1$ and $-V_2$ above by
the last two sentences of Section~\ref{sec:prel}, we can have a similar example for 
nonnegative perturbations, because $1/2<1$. 
Let us also remark that the sharp global Gaussian estimates may 
hold for $V(x_1,x_2)=V_1(x_1)V_2(x_2)$ but fail for $V_1$ or $V_2$. 
Indeed, it suffices to consider $V_1(x_1)=-\ind_{|x_1|<1}$ on $\RR^3$ and $V_2\equiv 1$ on $\RR$, and to apply Theorem~\ref{prop:most_extended}. We see that it is  the combined effect of the factors $V_1$ and $V_2$ that matters--as captured in Section~\ref{sec:gen_app}.

\section{Appendix}\label{sec:a}

Following \cite{MR2457489,MR3200161} we study and use the following functions
\begin{align*}
f(t)&=\sup_{x,y\in\Rd}S(V,t,x,y)\,,\qquad t\in(0,\infty),\\
F(t)&=\sup_{0<s<t}f(s)=\sup_{\substack{x,y\in\Rd\\ 0<s< t}}S(V,s,x,y)\,,\qquad t\in(0,\infty]\,.
\end{align*}
We fix $V$ and $x,y\in \Rd$. For $0<\varepsilon<t$, we consider
\begin{align*}
S(V,t-\varepsilon,x,y)=\int_0^t \int_{\Rd} \frac{g(s,x,z)g(t-\varepsilon-s,z,y)}{g(t-\varepsilon,x,y)}|V(z)|\,{\bf 1}_{[0,t-\varepsilon]}(u)\,dzds.
\end{align*}
By Fatou's lemma we get
$$
S(V,t,x,y)\le \liminf_{\varepsilon\to 0} S(V,t-\epsilon,x,y),
$$
meaning that $(0,\infty)\ni t\mapsto S(V,t,x,y)$ is lower semicontinuous on the left.
It follows that
$f$ is lower semi-continuous on the left, too.
In consequence, $f(t)\leq F(t)$ and $F(t)=\sup_{0<s\leq t} f(s)$ for $0<t<\infty$.

We next claim that $f$ is sub-additive, that is,
\begin{align}\label{ineq:S_chapm}
f(t_1+t_2)\leq f(t_1)+f(t_2)\,,\qquad t_1,\,t_2>0\,.
\end{align}
This follows from the Chapman-Kolmogorov equations for $g$.
Indeed, we have $S(V,t_1+t_2,x,y)=I_1+I_2$, where
\begin{align*}
I_1&=\int_0^{t_1} \int_{\Rd} \frac{g(s,x,z)g(t_1+t_2-s,z,y)}{g(t_1+t_2,x,y)}|V(z)|\,dzds\\
&=\int_0^{t_1} \int_\Rd\int_{\Rd} \frac{g(s,x,z)g(t_1-s,z,w)g(t_2,w,y)g(t_1,x,w)}{g(t_1+t_2,x,y)g(t_1,x,w)}|V(z)|\,dwdzds\\
&\le \int_\Rd \frac{g(t_2,w,y)g(t_1,x,w)}{g(t_1+t_2,x,y)} S(V,t_1,x,w)\,dw\le f(t_1)\,,
\end{align*}
and $I_2$ equals
\begin{align*}
&\int_{t_1}^{t_1+t_2} \int_{\Rd} \frac{g(s,x,z)g(t_1+t_2-s,z,y)}{g(t_1+t_2,x,y)}|V(z)|\,dzds
\\
&=\int_{t_1}^{t_1+t_2} \!\!\!\! \int_\Rd\int_{\Rd} \frac{g(t_1,x,w)g(s-t_1,w,z)g(t_2-(s-t_1),z,y)g(t_2,w,y)}{g(t_1+t_2,x,y)g(t_2,w,y)}|V(z)|\,dwdzds\\
&\le \int_\Rd \frac{g(t_1,x,w)g(t_2,w,y)}{g(t_1+t_2,x,y)} S(V,t_2,w,y)\,dw
\le f(t_2)\,.
\end{align*}
This yields \eqref{ineq:S_chapm}.
\begin{lem}\label{lemaF(t)}
For all $t,h>0$ we have
$f(t)
\leq
F(h)+ t\, f(h)/h.
$
\end{lem}
\begin{proof}
Let $k\in \NN$ be such that $(k-1)h<t\leq kh$, and let $\theta=t-(k-1)h$.
Then $t=\theta+(k-1)h$, and by
\eqref{ineq:S_chapm} we get
$$
f(t)\leq f(\theta)+ t\,f(h)/h \,\leq F(h)+t\, f(h)/h\,,
$$
since $0<\theta\leq h$.
\end{proof}
\begin{cor}\label{cor:ineq_most}
$F(t)\leq F(h)+ t\, F(h)/h$ and $F(2t)\le 2F(t)$ for $t,h>0$.
\end{cor}
\begin{proof}[Proof of Lemma~\ref{lem:gen_neg}]
The left-hand side of \eqref{gen_est} follows from \cite[pp. 467-469]{MR1978999} and Lemma~\ref{l:b}, or we can use
\cite[(41)]{MR2457489}, which follows therein from Jensen's inequality and the second displayed formula on page 252 of \cite{MR2457489}.
We now  prove the right hand side of \eqref{gen_est}. Since $G$ is increasing in $V$, we may assume that $V\ge 0$.  For $0<s<t$, $x,y\in\mathbb R^d$, we let $p_0(s,x,t,y)=g(t-s,x,y)$ and $p_n(s,x,t,y)=\int_s^t \int_{\mathbb R^d} p_{n-1}(s,x,u,z)V(z)p_0(u,z,t,y)\, dz\, du$, $n\in \NN$. Let $Q:\mathbb R\times \mathbb R\to [0,\infty)$ satisfy $Q(u,r)+Q(r,v)\leq Q(u,v)$. By \cite[Theorem 1]{MR2507445} (see also \cite[Theorem~3]{MR3000465}) if there is $0<\eta <1$ such that
\begin{equation}\label{condition1}
p_1(s,x,t,y)\leq [\eta + Q(s,t)]p_0(s,x,t,y),
\end{equation}
then
\begin{equation}\label{eq2}
\tilde p(s,x,t,y):=\sum_{n=0}^\infty p_n(s,x,t,y)\leq \Big(\frac{1}{1-\eta}\Big)^{1+\frac{Q(s,t)}{\eta}}p_0(s,x,t,y)\,. \end{equation}
Corollary \ref{cor:ineq_most} and the assumptions of the lemma imply that (\ref{condition1}) is satisfied with $
\eta=F(h)< 1$ and
$Q(s,t)=(t-s)F(h)/h
$. Since $G(t,x,y)=\tilde p(0,x,t,y)$, the proof of \eqref{gen_est} is complete (see also \cite[(17)]{MR2457489}).
\end{proof}
\begin{proof}[Proof of Lemma~\ref{cor:gen_neg}]
Let $V\le 0$. By the proof of Theorem~1.1(a)
at the bottom of p. 468 in \cite{MR1978999}, the boundedness of $S(V,t,x,y)$ is necessary and sufficient for \eqref{est:sharp_uni}.
In particular, by the displayed formula proceeding \cite[(3.1)]{MR1978999}
the boundedness of $S(V,t,x,y)$ is sufficient for \eqref{est:sharp_uni}.
Alternatively we can apply Jensen inequality to the second displayed formula on p. 252 in \cite{MR2457489}.
The first part of Lemma~\ref{cor:gen_neg} is proved.
The second part is obtained in the same way, by
restricting the considerations, and
the transition kernel, to bounded time interval.
\end{proof}

As a consequence of Corollary~\ref{cor:ineq_most} we obtain the following result.
\begin{cor}\label{prop:lower_exp}
Let $V\leq 0$ and $T>0$.
Then \eqref{est:sharp_time} holds
if and only if
\begin{align}\label{ineq:lower_exp}
C e^{-ct} g(t,x,y)\leq G(t,x,y) \,,\qquad t>0,\,x,y\in\Rd\,,
\end{align}
for some constants $C>0$, $c\geq 0$. In fact we can take
$$\ln C=-\sup_{\substack{x,y\in\Rd\\ 0<t\leq T}}S(V,t,x,y) \qquad {\rm and}\qquad c=\frac1T \sup_{x,y\in\Rd} S(V,T,x,y)\,.$$
\end{cor}
\begin{proof}
Obviously,
\eqref{ineq:lower_exp} implies \eqref{est:sharp_time} for every fixed $T>0$. Conversely, if \eqref{est:sharp_time} holds for fixed $T>0$,
then by Lemma~\ref{lem:gen_neg} and \ref{lemaF(t)} we have
$$\frac{G(t,x,y)}{g(t,x,y)}\geq
e^{-S(V,t,x,y)}\ge
e^{-f(t)}\geq e^{-F(T)} e^{-tf(T)\slash T}.$$
\end{proof}
\noindent
We note in passing that the above proof shows that \eqref{est:sharp_time} is determined by the behavior of $\sup_{x,y\in\Rd}S(V,t,x,y)$ for small $t>0$.
We end our discussion by recalling the connection of $G$ to $\Delta+V$ aforementioned in Abstract.
As it is well known, and can be directly checked by using the Fourier transform or by arguments of the semigroup theory \cite[Section~4]{BBS},
\begin{align*}
\int_s^{\infty}\int_{\Rd} g(u-s,x,z)\Big[\partial_u\phi(u,z)+\Delta \phi(u,z)\Big] dzdu=-\phi(s,x)\,,
\end{align*}
for all $s\in\RR$, $x\in\Rd$ and  for all $\phi\in C_c^{\infty}(\RR\times\Rd)$, the smooth compactly supported test functions on space-time.
Similarly, if $V$ satisfies the assumptions of Lemma~\ref{lem:gen_neg}, then by \cite[Theorem~1.1]{MR1978999}
for all $s\in\RR$, $x\in\Rd$, $\phi\in C_c^{\infty}(\RR\times\Rd)$,
\begin{align*}
\int_s^{\infty}\int_{\Rd} G(u-s,x,z)\Big[\partial_u\phi(u,z)+\Delta \phi(u,z) +V(z)\phi(u,z)\Big] dzdu=-\phi(s,x)\,.
\end{align*}
We refer to \cite[Lemma~4]{MR3000465} for a general approach to such identities.

\section*{Acknowledgement}
Jacek Dziuba{\'n}ski was supported by the Polish National Science Center
(Narodowe Centrum Nauki) grant DEC-2012/05/B/ST1/00672.
Karol Szczypkowski was partially supported by IP2012 018472
and by the German Science Foundation (SFB 701).

\def\polhk#1{\setbox0=\hbox{#1}{\ooalign{\hidewidth
  \lower1.5ex\hbox{`}\hidewidth\crcr\unhbox0}}}
  \def\polhk#1{\setbox0=\hbox{#1}{\ooalign{\hidewidth
  \lower1.5ex\hbox{`}\hidewidth\crcr\unhbox0}}}


\begin{thebibliography}{10}

\bibitem{MR644024}
M.~Aizenman and B.~Simon.
\newblock Brownian motion and {H}arnack inequality for {S}chr\"odinger
  operators.
\newblock {\em Comm. Pure Appl. Math.}, 35(2):209--273, 1982.

\bibitem{MR0435594}
D.~G. Aronson.
\newblock Non-negative solutions of linear parabolic equations.
\newblock {\em Ann. Scuola Norm. Sup. Pisa (3)}, 22:607--694, 1968.

\bibitem{BBS}
K.~Bogdan, Y.~Butko, and K.~Szczypkowski.
\newblock Majorization, {4G} {T}heorem and {S}chr\"odinger perturbations.
\newblock {\em Journal of Evolution Equations}, pages 1--20, 2015.

\bibitem{2016arXiv160603745B}
K.~{Bogdan}, J.~{Dziuba{\'n}ski}, and K.~{Szczypkowski}.
\newblock {Characterization of sharp global Gaussian estimates for
  Schr$\backslash$''odinger heat kernels}.
\newblock {\em ArXiv e-prints}, June 2016.

\bibitem{MR2457489}
K.~Bogdan, W.~Hansen, and T.~Jakubowski.
\newblock Time-dependent {S}chr\"odinger perturbations of transition densities.
\newblock {\em Studia Math.}, 189(3):235--254, 2008.

\bibitem{MR3000465}
K.~Bogdan, T.~Jakubowski, and S.~Sydor.
\newblock Estimates of perturbation series for kernels.
\newblock {\em J. Evol. Equ.}, 12(4):973--984, 2012.

\bibitem{MR3200161}
K.~Bogdan and K.~Szczypkowski.
\newblock Gaussian estimates for {S}chr\"odinger perturbations.
\newblock {\em Studia Math.}, 221(2):151--173, 2014.

\bibitem{MR549115}
R.~Carmona.
\newblock Regularity properties of {S}chr\"odinger and {D}irichlet semigroups.
\newblock {\em J. Funct. Anal.}, 33(3):259--296, 1979.

\bibitem{MR1329992}
K.~L. Chung and Z.~X. Zhao.
\newblock {\em From {B}rownian motion to {S}chr\"odinger's equation}, volume
  312 of {\em Grundlehren der Mathematischen Wissenschaften [Fundamental
  Principles of Mathematical Sciences]}.
\newblock Springer-Verlag, Berlin, 1995.

\bibitem{MR591851}
E.~B. Davies.
\newblock {\em One-parameter semigroups}, volume~15 of {\em London Mathematical
  Society Monographs}.
\newblock Academic Press Inc. [Harcourt Brace Jovanovich Publishers], London,
  1980.

\bibitem{MR2164260}
J.~Dziuba{\'n}ski and J.~Zienkiewicz.
\newblock Hardy spaces {$H^1$} for {S}chr\"odinger operators with compactly
  supported potentials.
\newblock {\em Ann. Mat. Pura Appl. (4)}, 184(3):315--326, 2005.

\bibitem{GS}
T.~{Grzywny} and K.~{Szczypkowski}.
\newblock {Kato classes for L{\'e}vy processes}.
\newblock {\em ArXiv e-prints}, Mar. 2015.

\bibitem{MR2253111}
A.~Gulisashvili and J.~A. van Casteren.
\newblock {\em Non-autonomous {K}ato classes and {F}eynman-{K}ac propagators}.
\newblock World Scientific Publishing Co. Pte. Ltd., Hackensack, NJ, 2006.

\bibitem{MR2507445}
T.~Jakubowski.
\newblock On combinatorics of {S}chr\"odinger perturbations.
\newblock {\em Potential Anal.}, 31(1):45--55, 2009.

\bibitem{MR1817225}
E.~H. Lieb and M.~Loss.
\newblock {\em Analysis}, volume~14 of {\em Graduate Studies in Mathematics}.
\newblock American Mathematical Society, Providence, RI, second edition, 2001.

\bibitem{MR1642818}
V.~Liskevich and Y.~Semenov.
\newblock Two-sided estimates of the heat kernel of the {S}chr\"odinger
  operator.
\newblock {\em Bull. London Math. Soc.}, 30(6):596--602, 1998.

\bibitem{MR1783642}
V.~Liskevich and Y.~Semenov.
\newblock Estimates for fundamental solutions of second-order parabolic
  equations.
\newblock {\em J. London Math. Soc. (2)}, 62(2):521--543, 2000.

\bibitem{MR2253015}
V.~Liskevich, H.~Vogt, and J.~Voigt.
\newblock Gaussian bounds for propagators perturbed by potentials.
\newblock {\em J. Funct. Anal.}, 238(1):245--277, 2006.

\bibitem{MR1994762}
P.~D. Milman and Y.~A. Semenov.
\newblock Heat kernel bounds and desingularizing weights.
\newblock {\em J. Funct. Anal.}, 202(1):1--24, 2003.

\bibitem{MR1687500}
R.~Schnaubelt and J.~Voigt.
\newblock The non-autonomous {K}ato class.
\newblock {\em Arch. Math. (Basel)}, 72(6):454--460, 1999.

\bibitem{MR670130}
B.~Simon.
\newblock Schr\"odinger semigroups.
\newblock {\em Bull. Amer. Math. Soc. (N.S.)}, 7(3):447--526, 1982.

\bibitem{MR1009389}
J.~A. van Casteren.
\newblock Pointwise inequalities for {S}chr\"odinger semigroups.
\newblock In {\em Semigroup theory and applications ({T}rieste, 1987)}, volume
  116 of {\em Lecture Notes in Pure and Appl. Math.}, pages 67--94. Dekker, New
  York, 1989.

\bibitem{MR845197}
J.~Voigt.
\newblock Absorption semigroups, their generators, and {S}chr\"odinger
  semigroups.
\newblock {\em J. Funct. Anal.}, 67(2):167--205, 1986.

\bibitem{MR1360232}
Q.~S. Zhang.
\newblock On a parabolic equation with a singular lower order term.
\newblock {\em Trans. Amer. Math. Soc.}, 348(7):2811--2844, 1996.

\bibitem{MR1457736}
Q.~S. Zhang.
\newblock Gaussian bounds for the fundamental solutions of {$\nabla (A\nabla
  u)+B\nabla u-u_t=0$}.
\newblock {\em Manuscripta Math.}, 93(3):381--390, 1997.

\bibitem{MR1488344}
Q.~S. Zhang.
\newblock On a parabolic equation with a singular lower order term. {II}. {T}he
  {G}aussian bounds.
\newblock {\em Indiana Univ. Math. J.}, 46(3):989--1020, 1997.

\bibitem{MR1978999}
Q.~S. Zhang.
\newblock A sharp comparison result concerning {S}chr\"odinger heat kernels.
\newblock {\em Bull. London Math. Soc.}, 35(4):461--472, 2003.

\end{thebibliography}
\end{document}